\documentclass[12pt]{amsart}
\usepackage{amsmath,amssymb,amsfonts,epsfig}
\usepackage{xcolor}
\usepackage{amsfonts}
\usepackage{amssymb}
\usepackage{graphicx}
\usepackage{pstricks}
\usepackage{amsmath}
\usepackage{amscd}
\usepackage{epsfig}
\usepackage{bm}

\usepackage{amsxtra}
\usepackage{etoolbox}
\usepackage{mathabx}
\usepackage{hyperref}
\usepackage{cleveref}

\usepackage{tkz-tab}

\usepackage[all]{xy}

\usepackage{cancel}
\usepackage{ulem}
\usepackage{xcolor}

\newtheorem{thm}{Theorem}[section]
\newtheorem{prop}[thm]{Proposition \!\!}
\newtheorem{cor}[thm]{Corollary \!\!}

\newtheorem{lem}[thm]{Lemma \!\!}

\newtheorem{remark}[thm]{Remark \!\!}
\newtheorem{conjecture}[thm]{Conjecture \!\!}
\newtheorem{definition}[thm]{Definition}

\newcommand\mk{\medskip}
\newcommand\bk{\bigskip}

\newcommand{\pfend}{\hfill $\Box$ \medskip}

\begin{document}


\title{Geometrically regular weighted shifts}
\author{Chafiq Benhida}
\address{UFR de Math\'{e}matiques, Universit\'{e} des Sciences et
Technologies de Lille, F-59655, Villeneuve-d'Ascq Cedex, France}
\email{chafiq.benhida@univ-lille.fr}

\author{Ra\'ul E. Curto}
\address{Department of Mathematics, University of Iowa, Iowa City, Iowa 52242-1419, USA}
\email{raul-curto@uiowa.edu}
\urladdr{http://www.math.uiowa.edu/\symbol{126}rcurto/}

\author{George R. Exner}
\address{Department of Mathematics, Bucknell University, Lewisburg, Pennsylvania 17837, USA}
\email{exner@bucknell.edu}



\begin{abstract}
We study a general class of weighted shifts whose weights $\alpha$ are given by $\alpha_n = \sqrt{\frac{p^n + N}{p^n + D}}$, where $p > 1$ and $N$ and $D$ are parameters so that $(N,D) \in (-1, 1)\times (-1, 1)$.  \ Some few examples of these shifts have appeared previously, usually as examples in connection with some property related to subnormality. \  In sectors nicely arranged in the unit square in $(N,D)$, we prove that these geometrically regular weighted shifts exhibit a wide variety of properties: moment infinitely divisible, subnormal, $k$-- but not $(k+1)$--hyponormal, or completely hyperexpansive, and with a variety of well-known functions (such as Bernstein functions) interpolating their weights squared or their moment sequences. \ They provide subshifts of the Bergman shift with geometric, not linear, spacing in the weights which are moment infinitely divisible. \ This new family of weighted shifts provides a useful addition to the library of shifts with which to explore new definitions and properties.\
\end{abstract}

\maketitle

\tableofcontents

\setcounter{tocdepth}{2}

\medskip
\section{Introduction and Statement of Main Results} \label{Intro}

Let $\mathcal{H}$ be a separable, complex infinite dimensional Hilbert space and $\mathcal{L}(\mathcal{H})$ denote the algebra of bounded linear operators on $\mathcal{H}$. \  Recall that an operator $T$ is \textit{subnormal} if it is the restriction to a (closed) invariant subspace of a normal operator, and \textit{hyponormal} if $T^* T \geq T T^*$. \ A frequently-studied class of operators is the class of unilateral weighted shifts $W_{\alpha}$ acting on the classical sequence space $\ell^2(\mathbb{N}_0)$. \  We define and study a family of such weighted shifts, the \textbf{geometrically regular weighted shifts (GRWS)}, with weight sequence
$$\alpha_n = \sqrt{\frac{p^n + N}{p^n + D}}, \quad n = 0, 1,2, \ldots$$
where $p > 1$ and $N$ and $D$ are constants. \   This family exhibits a wide variety of examples of the properties related to or extending subnormality;  in particular, some weighted shifts in this family are moment infinitely divisible ($\mathcal{MID}$ -- definitions reviewed below), some are ``merely'' subnormal, some are $k$--hyponormal for various positive integers $k$, some are completely hyperexpansive, and some fall into classes apparently not yet studied.   For a particular choice of $p$, the loci where  these various properties hold yield a geometrically pleasing diagram in the parameter space $(N,D)$ for these shifts. \

The motivation for their introduction is four-fold: \newline
(i) \ some choices of the pair $(N,D)$ provide new examples of $\mathcal{MID}$ shifts or subnormal shifts; \newline
(ii) \ particular examples of these shifts yield subshifts of the well-known Bergman shift with ``geometric'' spacing which are $\mathcal{MID}$, subnormal, or neither; \newline
(iii) \ instances of these shifts have appeared previously as examples or counter-examples when considering the Aluthge transform or Schur square root problem for shifts; \newline
(iv) \ finally, in the course of previous studies it has become clear to us that the example pool of standard families of weighted shifts is rather small, and, in particular, that the two most standard families (the shifts with finitely atomic Berger measures and the Agler shifts of which the Bergman shift is the best-known example) are so special as to be seriously misleading.

We turn next to setting notation and recalling definitions.

\medskip
\subsection{Unilateral weighted shifts}

We employ the standard notation for weighted shifts: let  $\ell^2$ be the classical Hilbert space of square summable complex sequences, with canonical orthonormal basis $e_0, e_1, e_2,  \ldots$ (note indexing begins at zero). \  Let $\alpha: \alpha_0, \alpha_1, \alpha_2, \ldots$ be a (bounded) positive \textbf{weight sequence} and  $W_\alpha$  the weighted shift defined by $W_\alpha e_j := \alpha_j e_{j+1} \;\; (j =0,1,2, \ldots)$ and extended by linearity. \ (While weighted shifts can be defined for any bounded sequence $\alpha$, without loss of generality for our questions of interest we can and do assume henceforth that $\alpha$ is positive.) \ The (unweighted) shift is the classical unilateral shift $U_+e_{j}:=e_{j+1} \; (j =0,1,\ldots )$ corresponding to the constant sequence in which $\alpha_n = 1$ for all $n$. \ The \textbf{moments} $\gamma = (\gamma_n)_{n=0}^\infty$ of the shift are defined by $\gamma_0 := 1$ and $\gamma_n := \prod_{j=0}^{n-1} \alpha_j^2$ for $n \geq 1$. \ It is well known from \cite[III.8.16]{Con} and \cite{GW} that a weighted shift $W_\alpha$ is subnormal if and only if it has a \textbf{Berger measure}, meaning a probability measure $\mu$ supported on $[0, \|W_\alpha\|^2]$ such that
$$
\gamma_n = \int_0^{\|W_\alpha\|^2} t^n d \mu(t), \hspace{.2in} n = 0, 1,2, \ldots .
$$

The canonical polar decomposition of $W_{\alpha}$ is $U_+ P_{\alpha}$, where $P_{\alpha}$ is the diagonal operator with diagonal entries $\alpha_0,\alpha_1,\alpha_2, \ldots$ . \ The \textbf{Aluthge transform} \cite{Al} of $W_{\alpha}$ is given by $AT(W_{\alpha}):=\sqrt{P_{\alpha}}U_+ \sqrt{P_{\alpha}}$, and this is the weighted shift with weight sequence $\sqrt{\alpha_0 \alpha_1},\sqrt{\alpha_1 \alpha_2},\sqrt{\alpha_2 \alpha_3},\ldots$ .

A \textbf{subshift} of a weighted shift $W_\alpha$ is a shift with weight sequence of the form $\alpha \circ g$, where $g :\mathbb{N}_0 \rightarrow \mathbb{N}_0$ is increasing and $\circ$ is composition. \  We say that the subshift is an \textbf{affine subshift} if $g$ is affine, of the form $g(n) = \ell n + r$, with $\ell$ a positive integer and $r$ a non-negative integer.

\medskip
\subsubsection{{\bf $k$--hyponormality}} \label{khypon} \ Recall that an operator $T$ is hyponormal if $T^*T-TT^* \ge 0$, and one computes easily that for a weighted shift $W_{\alpha}$ this is exactly $\alpha_n^2 \le \alpha_{n+1}^2$ for all $n =0,1,\ldots$. \   It is the Bram-Halmos characterization of subnormality (\cite{Br}) that an operator $T$ is subnormal if and only if, for every $k = 1, 2, \ldots$, a certain $(k+1) \times (k+1)$ operator matrix $A_n(T)$ is positive. \ For $k \ge 1$, an operator is $k$--hyponormal if this positivity condition holds for $k$. \ It is well known from \cite[Theorem 4]{Cu} that for weighted shifts $k$--hyponormality reduces to the positivity, for each $\ell$, of the $(k+1) \times (k+1)$ (Hankel) moment matrix $M_{\gamma}(\ell,k)$, where
$$
M_{\gamma}(\ell,k) = \left(
\begin{array}{cccc}
\gamma _{n} & \gamma _{n+1} & \cdots & \gamma _{n+k} \\
\gamma _{n+1} & \gamma _{n+2} & \cdots & \gamma _{n+k+1} \\
\vdots & \vdots & \ddots & \vdots \\
\gamma _{n+k} & \gamma _{n+k+1} & \cdots & \gamma _{n+2k}%
\end{array}
\right).
$$

\medskip
\subsubsection{{\bf $n$--contractivity, sequence and function monotonicity, and related properties}} \ For a contractive operator $T$ (that is, $\|T \| \leq 1$), there is another approach to subnormality, namely the Agler-Embry characterization based on the notion of $n$--contractivity. \  For $n \ge 1$, an operator is $n$\textbf{--contractive} if
\begin{equation} \label{eq11}
A_n(T):=\sum_{i=0}^n (-1)^i \binom{n}{i} {T^*}^i T^i \geq 0 \quad \; \textrm{(cf. \cite{Ag})}.
\end{equation}
From \cite{Ag} it is known that a contractive operator is subnormal if and only if it is $n$--contractive for all positive integers $n$. \ For a weighted shift one computes readily that it suffices to test this condition on basis vectors and that a weighted shift is $n$--contractive if and only if
\begin{equation} \label{cond100}
\sum_{i=0}^n (-1)^i \binom{n}{i} \gamma_{k+i} \geq 0, \qquad k = 0, 1,2, \ldots.
\end{equation}

Given a sequence $a = (a_j)_{j=0}^\infty$, let $\nabla$ (the \textbf{forward difference operator}) be defined by
$$
(\nabla a)_j := a_j - a_{j+1},
$$
and the iterated forward difference operators $\nabla^{n}$ by
$$
\nabla^{0} a := a \; \; \textrm{ and } \; \; \nabla^{n} := \nabla (\nabla^{n-1}),
$$
for $n \geq 1$. \ For example,
\begin{equation} \label{nabla2}
(\nabla^2a)_j=a_j-2a_{j+1}+a_{j+2} \; \; (j =0,1,\ldots).
\end{equation}

Using $\nabla$ and $\Delta := -\nabla$ there is alternative language for $n$--contractivity and related notions:  a sequence $a$ is $n$\textbf{--monotone} if $(\nabla^{n} a)_k \geq 0$ for all $k = 0, 1,2, \ldots$, $n$\textbf{--hypermonotone} if it is $j$--monotone for all $j = 1, \ldots, n$, and \textbf{completely monotone} if it is $n$--monotone for all $n = 1, 2, \ldots$. \  \label{defn} As well, a sequence is $n$\textbf{--alternating} if $(\nabla^{n} a)_k \leq 0$ for all $k = 0, 1,2, \ldots$, $n$\textbf{--hyperalternating} if it is $j$--alternating for all $j = 1, \ldots, n$, and \textbf{completely alternating} if it is $n$--alternating for all $n = 1, 2, \ldots$. \  A sequence is $n$\textbf{--log monotone} (respectively, \textbf{completely log monotone}, $n$\textbf{--log alternating}, \textbf{completely log alternating}) if the sequence $(\ln a_j)$ is $n$--monotone (respectively, completely monotone, $n$--alternating, completely alternating). \ (Notice that we define log monotonicity using $n$--monotonicity, so we allow for the possibility that the sequence is negative, as in the case of $\ln a$, where $a$ is the sequence of moments of a contractive weighted shift; thus, the sequence $a$ is $n${--log monotone} is not quite that the sequence $\ln (a)$ is $n${--monotone}. \ As a result, we differ slightly from the more restrictive definition given in \cite[Definition 6.1]{BCR}.)  It is shown in \cite{BCE1} that if a sequence is completely alternating, it is log completely alternating, and that the reverse is not true in general.

There are strongly related analogs of the properties above for functions. \  A function $f : \mathbb{R}_+ \rightarrow \mathbb{R}_+$ is \textbf{completely monotone} if its derivatives alternate in sign so that $f^{(2n+1)} \leq 0$ ($n = 0, 1,2, \ldots$) and $f^{(2n)} \geq 0$ ($n =  1,2, \ldots$). \  A function $f : \mathbb{R}_+ \rightarrow \mathbb{R}_+$ is a \textbf{Bernstein} function if $f'$ is completely monotone. \  It is immediate that if $g = 1 - f$, with $g$ nonnegative and $0 \le f \le 1$ completely monotone, then $g$ is Bernstein. \ We define $f$ to be log completely monotone [respectively, log Bernstein] if $\ln f$ is completely monotone [respectively, Bernstein] but again with the relaxation that we do not require $\ln f$ to be nonnegative, but insist only on the appropriate alternation of signs of the derivative. \  It is well known that sequences interpolated by one of these classes of functions inherit the analogous sequence property;  for example, if a moment sequence is interpolated by a log completely monotone function, then the sequence is log completely monotone (and the corresponding shift is thereby $\mathcal{MID}$ as discussed below). \  We will sometimes refer to a function which satisfies the alternating derivatives condition to be a Bernstein function, but which is potentially negative on $\mathbb{R}_+$, as a \textbf{generalized} Bernstein function.

\medskip
\subsubsection{{\bf Moment infinitely divisible ($\mathcal{MID}$) shifts}} \ In \cite{BCE1} and \cite{BCE3} the authors considered a class of ``better than subnormal'' weighted shifts. \  We say that a shift $W_\alpha$ is \textbf{moment infinitely divisible} ($\mathcal{MID}$) if, for every $s > 0$, the ``Schur $s$--th power'' shift with weight sequence $(\alpha_n^{s})$ is subnormal. \  It is easily seen that raising weights to the power $s$ and raising moments to the power $s$ are equivalent. \  Examples of $\mathcal{MID}$ shifts include the Agler shifts $A_j$, $j = 1, 2, \ldots$ (where $A_j$ has the weight sequence $\sqrt{\frac{n+1}{n+j}}$), with $j=1$ yielding the unweighted shift and $j=2$ the familiar Bergman shift \cite{BCE1}. \  This can be generalized to the so-called homographic shifts $S(a,b,c,d)$ with weight sequence $\sqrt{\frac{a n + b}{c n + d}}$ where $a$, $b$, $c$, and $d$ are positive real numbers such that $ad - bc > 0$  (see \cite{CPY} and, for example, \cite{CD}). \  Except in very special cases, weighted shifts with finitely atomic Berger measures are not ($\mathcal{MID}$);  indeed, from arguments concerning the support of the measures, they will lack subnormality for some $s = \frac{1}{n}$ (that is, lack some Schur $n$--th root;  see \cite{SS} and \cite{CE}). \  The contractive $\mathcal{MID}$ shifts are those with weights squared log completely alternating or, equivalently, with moments log completely monotone (\cite{BCE1} and \cite{BCE3} respectively). \  It follows that if the weights squared sequence is interpolated by a Bernstein function (and so is completely alternating) the weights squared sequence is log completely alternating, and therefore the shift is $\mathcal{MID}$; however, we do not know what operator theoretic properties are improved by this stronger condition.

The $\mathcal{MID}$ shifts are robust under a variety of transformations; the results that follow are from \cite{BCE1} and \cite{BCE3}. \  One may generate the shift with weights quotients of successive weights of the original and it remains $\mathcal{MID}$, and the Aluthge transform maps the class $\mathcal{MID}$ bijectively to itself. \  The class is stable under appropriate limits of weight sequences or of Berger measures. \  Affine subshifts of $\mathcal{MID}$ are $\mathcal{MID}$. \  It is known that the shift with weights the reciprocals of those of a completely hyperexpansive shift is subnormal (\cite{At});  in fact such a shift is contractive $\mathcal{MID}$. \  (Taking the reciprocal of the weight sequence of a contractive subnormal (or even $\mathcal{MID}$) shift need not be completely hyperexpansive per the example in \cite{At}, and it remains open how to distinguish when this transformation does and does not work to produce a completely hyperexpansive shift.)  We note, although we do not pursue it further in this paper, that every example of a $\mathcal{MID}$ shift yields, via the $k$--hyponormality characterization of subnormality, a family of infinitely divisible matrices in the sense of \cite{Bh}.

\medskip
\subsection{Geometrically regular weighted shifts (GRWS)}

Here, we record only the definition;  motivation for and discussion of the definition are to be found in Section \ref{se:2}.

\begin{definition}
For $p > 0$, $-1 < N < 1$ and $-1 < D < 1$, we define the \textbf{geometrically regular weighted shift} with parameters $p$ and $(N,D)$ (often \textbf{GRWS}) to be the weighted shift with weight sequence $\alpha$, where
$$\alpha_n = \sqrt{\frac{p^n + N}{p^n + D}}, \quad n = 0, 1,2, \ldots\, .$$
\end{definition}

\medskip
\subsection{Statement of main results}

We state results for geometrically regular weighted shifts in reference to the following diagram, consisting of the open unit square in $\mathbb{R}^2$ with parameter pair $(N,D)$, and with Sectors I, II, \ldots, the ray $D = p N$ in Sector VIII and the rays of the form $D = p^n N$ in Sector IV. \ (Hereafter, by a sector we mean the convex cone in the open unit square with vertex at $(0,0)$ and bounded by two rays emanating from the center. \ All sectors in this paper will be closed in the relative topology of the open unit square.)

\bk
\begin{tikzpicture}[scale=4]

\draw[->] (-1.3,0) -- (1.3,0) ;

\draw[->] (0,-1.3) -- (0,1.3) ;

\draw  [-, dashed] (-1,-1)--(1,-1) --(1,1)--(-1,1)--cycle  ;

\draw (1,0) node[below right] {1} ;

\draw (-1,0) node[below left] {-1} ;

\draw (0,1) node[above left ] {1} ;

\draw (0,-1) node[below left] {-1} ;


\draw [][  domain=-1:1] plot(\x,{\x })node[right, scale=0.8] {$D=N$} ;

\draw [][  domain=-1:1] plot(\x,{-\x })node[right, scale=0.8] {$D=-N$} ;

\draw [blue ][  domain=-2/3:2/3] plot(\x,{3/2*\x })node[above right, scale=0.65] {$D=pN$} ;

\draw [-,blue, dashed][  domain=0:4/9] plot(\x,{9/4*\x })node[above , scale=0.6] {$D=p^2N$} ;

\draw [-, blue, dashed][  domain=0:8/27] plot(\x,{27/8*\x })node[above, scale=0.8] {} ;

\draw [-, blue, dashed][  domain=0:16/81] plot(\x,{81/16*\x })node[above, scale=0.8] {} ;

\fill[color=red!80] (0,0) -- (-1,-1) -- (-2/3,-1)  -- cycle ;

\fill[color=green] (0,0) -- (-1,-1) -- (-1,0)  -- cycle ;

\fill[green!40] (0,0) -- (-1,0) -- (-1,1)  -- cycle ;

\fill[blue!75] (0,0) -- (0,1) -- (-1,1)  -- cycle ;

\draw [green] (0,0) -- (1,1);

\draw (-0.6, -0.3) node [ ] {I};

\draw (-0.6, 0.3) node [ ] {II};

\draw (-0.3, 0.6) node [ ] {III};

\draw (0.3, 0.6) node [ ] {IV};

\draw (0.6, 0.3) node [ ] {V};

\draw (0.6, -0.3) node [ ] {VI};

\draw (0.3, -0.6) node [ ] {VII};

\draw (-0.3, -0.6) node [ ] {VIII};

\draw (-0.8, -0.9) node [ ] {A};

\draw (0, -1.5) node[below  ] {Fig. 1};

\end{tikzpicture}



\bk
We remark that the condition $-1 < N$ and $-1 < D$ avoids difficulties with the weight when $n = 0$, and it is convenient to require as well $N < 1$ and $D < 1$.

\begin{thm}
Let $p > 1$ and $(N, D)$ such that $-1 < N < 1$ and $-1 < D < 1$. \  Let $W_\alpha$ be a geometrically regular weighted shift with parameters $p$ and $(N,D)$; that is,
$$\alpha_n :=  \sqrt{\frac{p^n + N}{p^n + D}}, \quad n = 0, 1,2, \ldots .$$
Then $W_\alpha$ has the following properties in the sectors given in the diagram:
\begin{enumerate}
      \item Along the main diagonal $D=N$, $W_\alpha$ is the (unweighted) unilateral shift and thereby both $\mathcal{MID}$ and completely hyperexpansive;
     \mk \item In Sector {\rm I}, the weight squared sequence $\alpha$ is interpolated by a Bernstein function and $W_\alpha$ is $\mathcal{MID}$;
      \mk \item In Sector {\rm II}, the weight squared sequence $\alpha$ is interpolated by a log Bernstein function and $W_\alpha$ is $\mathcal{MID}$, and at least one shift in this sector does not have $\alpha$ interpolated by a Bernstein function;
      \mk \item In Sector {\rm III}, $W_\alpha$ is subnormal, and at least one shift in this sector is not $\mathcal{MID}$;
      \mk \item In Sector {\rm IV}, $W_\alpha$ is subnormal with finitely atomic Berger measure along the rays $D = p^k N$, $k = 0, 1, 2, \ldots$, and is not $\mathcal{MID}$ except for $N=D$;
      \mk \item In Sector {\rm VIIIA}, $W_\alpha$ is completely hyperexpansive.
    \end{enumerate}
\end{thm}

We note as well the following conjectures as suggested by the above.

\begin{conjecture}   \label{conjectures}
Let $p > 1$ and $(N, D)$ such that $-1 < N < 1$ and $-1 < D < 1$. \  Let $W_\alpha$ be a weighted shift with weight sequence $\alpha$, where
$$\alpha_n :=  \sqrt{\frac{p^n + N}{p^n + D}}, \quad n = 0, 1,2, \ldots .$$
Then $W_\alpha$ has the following properties in the sectors given in the diagram:
\begin{enumerate}
      \item In Sector {\rm II}, except on the boundary with Sector {\rm I}, no $\alpha$ is interpolated by a Bernstein function;
      \item In Sector {\rm III}, $W_\alpha$ is subnormal, and except on the boundary with Sector {\rm II}, no shift is $\mathcal{MID}$;
      \item Other than as specified above, there are no $\mathcal{MID}$, subnormal, or completely hyperexpansive shifts in the square.
    \end{enumerate}
\end{conjecture}

We have the following result concerning subshifts of the Agler shifts with geometrical spacing.

\begin{cor}   \label{cor:subshifts}
Let $p>1$ be a real number, let $N$ be an integer, let $j$ be an integer larger than $1$, and let $K$ be a positive integer such that $K > \max(|N|, |N+j-1|)$. \ Let $W_\alpha$ be the shift with weight sequence
$$\alpha_n = \sqrt{\frac{K p^n + N}{K p^n + N+j-1}}, \quad n = 0, 1,2, \ldots .$$
It is clear that $W_\alpha$ is a subshift of $A_j$, the $j$--th Agler shift. \ Further,
\begin{enumerate}
        \item If $N \le 1-j$ then $W_\alpha$ is $\mathcal{MID}$ with weights squared interpolated by a Bernstein function;
        \item If $1-j<N \le \dfrac{1-j}{2}$ then $W_\alpha$ is $\mathcal{MID}$ with weights squared interpolated by a log Bernstein function;
        \item If $\dfrac{1-j}{2}<N \le 0$ then $W_\alpha$ is subnormal;
        \item If $N = \dfrac{j-1}{p^k-1}$ for some integer $k =1, 2, \ldots$, then $W_\alpha$ is finitely atomic subnormal and not $\mathcal{MID}$.
    \end{enumerate}
\end{cor}

As well, we show that Sector IV is divided into subsectors of $k$--hyponormality but not $(k+1)$--hyponormality by the rays $D = p^k N$.

\begin{thm}   \label{kHNthm}
Let $p > 1$ and $(N, D)$ such that $0 < N < D < 1$ (that is, the point lies in Sector {\rm IV}). \   Let $W_\alpha$ be a weighted shift with weight sequence $\alpha$ where
$$\alpha_n :=  \sqrt{\frac{p^n + N}{p^n + D}}, \quad n = 0, 1,2, \ldots .$$
In the interior of  each  subsector bounded by $D = p^k N$ and $D = p^{k-1} N$ (for $k =  1, 2, \ldots$) the shift $W_\alpha$ is $k$--hyponormal but not $(k+1)$--hyponormal.
\end{thm}

\medskip

Finally, we show that the GRWS from the special lines in Sector IV provide a collection of examples of subnormal shifts whose affine subshifts need not be subnormal. Motivated by the fact that affine subshifts of the GRWS in Sectors I and II have their subshifts $\mathcal{MID}$  we prove in general that $\mathcal{MID}$ shifts have their affine subshifts $\mathcal{MID}$:

\begin{thm}\label{th:linearsubshiftofMID}
Let $W_\alpha$ be an $\mathcal{MID}$ shift. Then any affine subshift $W_{\alpha\circ g}$ is $\mathcal{MID}$.
\end{thm}

\medskip

\section{Proofs and Discussion}  \label{se:2}

We begin with noting the appearance of isolated examples of GRWS operators, usually in connection with subnormality or its relatives, and sometimes in conjunction with the notion of various sorts of subshifts of familiar shifts.  It is easy to show that affine subshifts of the Agler shifts are subnormal (and in fact $\mathcal{MID}$) because they fall within the class of homographic shifts.  One can easily construct, starting with a finitely atomic measure, a subnormal shift some of whose affine subshifts are not subnormal.  We note in the final subsection that the analysis of the GRWS provides other, and less \textit{ad hoc}, examples of subnormal shifts whose subshifts need not be subnormal.  We also show there that affine subshifts of $\mathcal{MID}$ are in fact $\mathcal{MID}$.  However, there is no particular reason to believe that other -- non-affine -- subshifts of even $\mathcal{MID}$ shifts should retain subnormality or any other desirable property stronger than hyponormality.

In \cite{CE} it was noted that the shift with weight sequence
$$\alpha_n = \sqrt{\frac{2^{n+2}-2}{2^{n+2}-1}}$$
has a subnormal Schur square root via generating functions, and the argument actually shows the shift is $\mathcal{MID}$. \  The weight sequence begins
$$\sqrt{\frac{2}{3}}, \sqrt{\frac{6}{7}}, \sqrt{\frac{14}{15}}, \sqrt{\frac{30}{31}}, \ldots $$
and this is clearly a subshift, although not an affine subshift, of the Bergman shift.

Further, in \cite{CPY} (and used later in \cite{Ex2} as an example of a subnormal shift whose Aluthge transform is not subnormal), there is the shift with weight sequence
$$
\alpha_n = \sqrt{\frac{2^{n+2}+1}{2^{n+2}+2}} \quad (n =0,1,\ldots).
$$
This shift is subnormal but not $\mathcal{MID}$ (its Berger measure is three-atomic), and is again clearly a subshift of the Bergman shift. \  This pair of apparently rather similar examples, but with quite different properties,  suggests something worthy of exploration. \  Per G. Polya's observation in \cite{Po}, it turns out to be sensible to consider a more general problem, in which we temporarily abandon conditions producing a subshift.

A further motivation for consideration of this class of shifts is that in the studies yielding \cite{CE}, \cite{BCE1}, \cite{BCE2}, and \cite{BCE3} we became all too aware that the families of shifts often studied (those with finitely atomic Berger measures -- in particular, the two-atomic shifts used by Stampfli in \cite{St} -- and the Agler shifts $A_j$ with weight sequences $\alpha_n = \sqrt{\frac{n+1}{n+j}}$\,) are sufficiently special to be seriously misleading. \  One often embarks upon the study of a new operator property by considering weighted shifts, and it is therefore prudent to enrich the class of tractable examples when possible.

\medskip
\subsection{Sectors {\rm I} and {\rm II}}

The results will follow from the point of view of function interpolation, and we begin with a definition and a lemma.

\begin{definition}  Given a family $G \equiv \{g_i\}_{i=1}^{\ell}$ of infinitely differentiable positive functions defined on $\mathbb{R}_+$, we say that $G$ is a \textbf{completely monotone family} if for each $k=1,2,\ldots$, the derivative $\frac{d}{d x} g_k(x)$ is a sum of terms of the form
\begin{equation}     \label{eq:form}
(-1) c_k f_k(x) \prod_{i=1}^{\ell} g_i(x)^{n_i}
\end{equation}
where the $n_i$ are nonnegative integers, $c_k$ is a positive constant, and $f_k$ is a completely monotone function ($1 \le k \le \ell$). \ (We interpret any $g_i(x)^0$ as the function identically $1$.)
\end{definition}

The name is justified by the following lemma.

\begin{lem}
Let $G \equiv \{g_i\}_{i=1}^{\ell}$ be a completely monotone family. \  Then each member of $G$ is completely monotone.
\end{lem}

\begin{proof} \  Without loss of generality, we show, by induction, that $g_1$ is completely monotone. \  In fact, we show that the $n$--th derivative of $g_1$ is a sum of terms of the form
\begin{equation}  \label{eq:derivform}
(-1)^n c_1^{(n)} f_1^{(n)}(x) \prod_{i=1}^{\ell} g_i(x)^{n_i(n)}
\end{equation}
where $c_1^{(n)}$ is a positive constant, $f_n$ is a completely monotone function, and the $n_i(n)$ are nonnegative integers. \  The case $n=1$ is precisely what is required by the definition in \eqref{eq:form}. \  Suppose then that the result holds for $n$. \  Each term in the $(n+1)$--st derivative of $g_1$ arises from the product rule applied to some term as in \eqref{eq:derivform}, and there are two possibilities. \  If the derivative of some $g_j(x)^{n_j(n)}$ (where we ignore the trivial case $n_j(n)=0$) is being taken as part of the product rule, what results is
\begin{eqnarray*}
&&(-1)^n c_1^{(n)} f_1^{(n)}(x) \prod_{i=1,i \neq j }^{\ell} g_i(x)^{n_i(n)} \cdot ({n_j(n)}-1)g_j(x)^{n_j(n)-1} g_j'(x) \\ &=& (-1)^n c_1^{(n)} f_1^{(n)}(x) \prod_{i=1,i \neq j }^{\ell} g_i(x)^{n_i(n)} \cdot  ({n_j(n)}-1)g_j(x)^{n_j(n)-1} \cdot (-1) c_j f_j(x) \prod_{i=1}^{\ell} g_i(x)^{n_j} \\
&=&(-1)^{n+1}  c_1^{(n)} c_j ({n_j(n)}-1) f_1^{(n)}(x) f_j(x) \prod_{i=1,i \neq j }^{\ell} g_i(x)^{n_i(n)} \cdot  \prod_{i=1}^
{\ell} g_i(x)^{n_j} \\
\end{eqnarray*}
and it is easy to rewrite this to see that it is of the form required for \eqref{eq:derivform}. \  (Of course, we use that the product of the completely monotone functions $f_1^{(n)}$ and $f_j$ is completely monotone.)

In the case in which the derivative of $f_1^{(n)}$ is being taken as part of the product rule, we use that $\frac{d}{d x} f_1^{(n)}$ is the negative of a completely monotone function since $f_1^{(n)}$ is completely monotone. \  We note that in each case the sign of the $(n+1)$--st derivative is as required, and the form is as claimed; this completes the induction.
\end{proof}

We have the following result.

\begin{thm}\label{th:1ovpxmaisCM}
Let $ 0<a<1$ and $p>1$, and let $f:\mathbb{R}_+ \rightarrow \mathbb{R}_+$ be the function defined by $f(x):
=\frac{1}{p^x-a}$. \ Then $f$ is completely monotone.
\end{thm}

\begin{proof} \ It is a straightforward computation to show that the family $\{\frac{1}{p^x - a}, \frac{p^x}{p^x - a}\}$ is a completely monotone family.
\end{proof}

\medskip

After a computational lemma, we may deduce another result about what will be a function interpolating a weights squared sequence.

\begin{lem}  \label{le:afnisCM}
Let $p > 1$, suppose $c$ satisfies $0 < c <1$, and consider the function $f$ defined by
$$f(x) = \frac{p^x \ln p}{p^x -c}, \quad x \in \mathbb{R}_+.$$
Then  the derivatives of $f$ have the form
\begin{equation}  \label{eq:formfn}
f^{(n)}(x) = \frac{(\ln p)^{n+1}\sum_{i=1}^n d^{(n)}_i c^{n+1-i} (p^x)^i}{(p^x-c)^{n+1}},
\end{equation}
where we do not specify the coefficients $d^{(n)}_i$ further except to note that they are independent of $c$ and all of the same sign (all negative if $n$ is odd and all positive if $n$ is even). \  It results in particular that $f$ is completely monotone.
\end{lem}

\begin{proof} \ The proof will be by induction, and one first computes that
$$f'(x) = \frac{-c p^x (\ln p)^2}{(p^x-c)^2}$$
as required for $n=1$. \  Suppose that, for some $n$, $f^{(n)}(x)$ has the form in \eqref{eq:formfn}. \  It is a computation using the quotient rule to show that $f^{(n+1)}(x)$ has the required form, with
$$
d^{n+1}_i = \left\{
\begin{array}{cc}
- d^{(n)}_1 c^{n+1}, & i=1, \\
& \\
\left(-i d^{(n)}_i -(n+1-i)d^{(n)}_{i-1}\right) c^{(n+2- i)}, & 1 \leq i \leq n, \\
& \\
- d^{(n)}_n c , & i = n+1.
\end{array}\right.
$$
From the signs shown above, the result follows.
\end{proof}

\begin{thm}\label{th:pxpaovpxpblogB}
Let $-1 < a < 0$, $0 < b < 1$, $b \leq -a$, and $p >1$, and let $f$ be the function defined by $f(x):=\frac{p^x+a}{p^x+b}$. \ Then $f$ is a log Bernstein function.
\end{thm}

\begin{proof} \ It clearly suffices to show that
$$g(x) := \ln \left(\frac{p^{x} +a}{p^{x}+b}\right) = \ln (p^{x} +a) - \ln (p^{x}+b)$$
is a (generalized) Bernstein function. \ One computes that
$$
f(x) := f_a(x) - f_b(x) = g'(x) = \frac{p^x \ln p}{p^x +a} - \frac{p^x \ln p}{p^x +b};
$$
observe that $f$ is positive on $\mathbb{R}_+$ and that it suffices to show that it is completely monotone.

Citing the lemma, we have that the derivatives of $f_a$ have the form in \eqref{eq:formfn} with $c$ of that lemma set to $-a$ (appropriately positive) and that the coefficients in its derivatives -- recall that these are all of the same sign -- are appropriately positive or negative to assure that it is completely monotone. \  By an argument similar to that in the proof of the lemma, we may show that the derivatives of $f_b$ have the form in \eqref{eq:formfn} with the same coefficients except that $c$ of that lemma is replaced by $-b$. \  (Of course since $-b < 0$, some of these are, and some are not, of the correct sign to make $f_b$ itself completely monotone.)  Consider now what happens if we find the $n$--th derivative of $f$ by finding the common denominator for the difference of the  $n$--th derivatives of $f_a$ and $f_b$:  the coefficient of some $(p^x)^i$ in the numerator of this derivative will have the form
$$(p^x + b)^{n+1} d^{(n)}_i (-a)^{n+1-i} - (p^x + a)^{n+1} d^{(n)}_i (-b)^{n+1-i}.$$
We claim that this will have the same sign as that of $d^{(n)}_i (-a)^{n+1-i}$, and this holds irrespective of the power applied to $(-b)$ because (recalling the signs of $a$ and $b$) $(p^x + b) > (p^x + a)$ and $a \geq -b$, so the first term just above dominates the second. \  Thus the sign of $f^{(n)}$ will be the same as the sign of $f_a^{(n)}$; it follows that $f$ is completely monotone. \  This completes the proof.
\end{proof}

\medskip

As a consequence, we obtain the following

\begin{cor}\label{cor:SecIBSecIIlogB} 
With the same notation as above, we have

\begin{enumerate}
\item If $(N,D) \in (-1,0]\times (-1,0]$  and $D\geq N$ then $\alpha^2(N,D)$ is interpolated by a Bernstein function.
\item
If $(N,D) \in (-1,0]\times [0,1)$  and $D\leq -N$ then $\alpha^2(N,D)$ is interpolated by a log Bernstein function.
\end{enumerate}
\end{cor}

\begin{proof} \  For the first claim, we may write the weights squared in the form
$$\alpha_n^2 = \frac{p^x + N}{p^x + D} = 1 -\frac{-N+D}{p^x + D}.$$
Noting $-N + D > 0$, and citing \Cref{th:1ovpxmaisCM}, we have that $\frac{-N+D}{p^x + D}$ is a completely monotone function and therefore that $\frac{p^x + N}{p^x + D}$ is a Bernstein function, as desired.

For the second claim, we may cite \Cref{th:pxpaovpxpblogB} to obtain the result.
\end{proof}

\medskip

Graphical and numerical experiments with \textit{Mathematica} \cite{Wol} strongly suggest that shifts arising from points $(N,D)$ in Sector II do not have weights squared sequences interpolated by a Bernstein function;  indeed, their weights squared sequences are apparently not completely alternating but merely log completely alternating. \ As one might expect, it takes higher order tests to discard a point as one becomes closer to the negative $N$ axis and thus approaches Sector I.

\medskip
\subsection{Sectors {\rm III} and {\rm IV}}

In distinction to the Agler shifts, for which both the weights and the moment sequences are relatively tractable (in the case of the moment sequence because there is useful cancellation of weights), the geometrically regular weighted shifts have simple weights but not particularly simple moments. \  To show that the shifts in Sector III are subnormal, as well as the shifts in Sector IV corresponding to points on the lines $D = p^k N$, we take a Berger measure approach.

We first consider Sector III.

\begin{prop}   \label{prop:measureSectorIII}
Let the weights squared be $\alpha_n = \sqrt{\frac{p^n + N}{p^n + D}} \;\, (n =0,1,\ldots),$ where we assume $p > 1$,$-1 < N \leq 0$, and $-N < D < 1$, with the associated moment sequence $\gamma$. \  Set \, $m_0 := 1$ and
$$m_i := \frac{p(D - p^{i-1}N)}{p^i - 1}, \quad i = 1, 2, \ldots.$$
Define $c_n$ for $n = 0, 1, 2, \ldots$ by
$$c_n := \prod_{i=0}^n m_i,$$
and set
$$a(N,D) := \frac{1}{\sum_{n=0}^\infty c_n}.$$
Then the measure $\mu$ defined by
$$\mu := a(N,D)\left(\sum_{i=0}^\infty c_i \delta_{\frac{1}{p^i}}\right)$$
is representing for $\gamma$. \  It satisfies $\int 1 \, d \mu = 1$, and thus is a probability measure.
\end{prop}

\begin{proof} \  First observe that under assumptions on $N$ and $D$ each $m_i$ is positive. \  The assertion that $\mu$ is a finite measure follows readily from the Ratio Test, and then that $\mu$ is a probability measure is immediate from the definition of $(N,D)$. \  To show that $\mu$ represents the sequence $\gamma$, it is enough to show that the moment sequence it does yield (say, $\hat{\gamma}$) satisfies $\hat{\gamma}_0 = 1$ and
$$\alpha_n^2 = \frac{\hat{\gamma}_{n+1}}{\hat{\gamma}_{n}}, \quad n = 0, 1, 2, \ldots .$$
The first has been shown above. \  The general condition is that for all $n = 0, 1, 2, \ldots$,
$$ \frac{p^n + N}{p^n + D} = \frac{a(N,D)\left[1 + c_1 \left(\frac{1}{p}\right)^{n+1} + c_1 \left(\frac{1}{p^2}\right)^{n+1} + \ldots\right]}{a(N,D)\left[1 + c_1 \left(\frac{1}{p}\right)^{n} + c_1 \left(\frac{1}{p^2}\right)^{n} + \ldots\right]}.$$
Clearing the denominators, we seek
$$(p^n + N)\left[1 + c_1 \left(\frac{1}{p}\right)^{n} + c_1 \left(\frac{1}{p^2}\right)^{n} + \ldots\right] =
(p^n + D)\left[1 + c_1 \left(\frac{1}{p}\right)^{n+1} + c_1 \left(\frac{1}{p^2}\right)^{n+1} + \ldots\right].$$
One checks readily that the coefficients of the $p^n$ term on each side agree. \  To have the constant ($p^0$) terms match, it is straightforward to check that we require
$$N + c_1 = D + c_1\frac{1}{p},$$
and solving for $c_1$ yields
$$c_1 = \frac{p(D-N)}{p-1},$$
as required. \newline
In general, for $i \geq 1$, to have the terms with $\left(\frac{1}{p^i}\right)^n$ match on the two sides clearly becomes
$$N c_i \left(\frac{1}{p^i}\right)^n + c_{i+1}\left(\frac{1}{p^{i+1}}\right)^n p^n = D c_i\left(\frac{1}{p^i}\right)^n + c_{i+1} \frac{1}{p^{i+1}}\left(\frac{1}{p^i}\right)^n,$$
and this clearly shows that we need
$$c_{i+1} = \frac{p c_i (D - p^i N)}{p^{i+1} -1}.$$
This, with an induction, yields the $m_i$, $c_n$, and measure $\mu$ as claimed.
\end{proof}

The theorem below is then immediate because the existence of a Berger measure guarantees subnormality, and examination of the terms $m_i$ show that they are strictly positive.

\begin{thm}
With the notation as above, points $(N,D)$ in Sector {\rm III} yield subnormal weighted shifts with countably atomic Berger measures.
\end{thm}

\begin{proof} \  Consideration of the terms $m_i$ shows that they are strictly positive;  the ratio test ensures that the measure is finite, and the support is clearly countable.
\end{proof}

\medskip

Again experiments, numerical and graphical, with \textit{Mathematica} \cite{Wol} strongly suggest that the weights squared sequences arising from Sector III are not log completely alternating (let alone interpolated by a log Bernstein function), and therefore the shifts are subnormal but not $\mathcal{MID}$.

The same measure-theoretic approach may be used on the special lines $D = p^k N$ in Sector IV. \  It is a computation to check that nothing in the identification of the measure as above is altered for these pairs $(N,D)$.

\begin{thm}  \label{th:fintomicSIV}
Consider some pair $(N,D)$ on the ray $D = p^k N$ in Sector {\rm IV} for some $k =  1,2, \ldots$ and with $D > 0$. \  Then $m_i > 0$, $i = 0, 1,2, \ldots, k$, $m_{k+1} = 0$, and thus $c_i > 0$, $i = 0, 1,2, \ldots, k$, and $c_i = 0$, $i \geq k+1$. \  The associated weighted shift then has a $(k+1)$--atomic Berger measure and is subnormal but not $\mathcal{MID}$.
\end{thm}

\begin{proof} \  One computes readily that $m_i > 0$, $i = 0, 1, \ldots, k$, and $m_{k+1}= 0$, yielding the $c_i$ as claimed. \  Support arguments as in \cite{SS} and \cite{CE} show that the corresponding shift cannot be $\mathcal{MID}$.
\end{proof}

\medskip

Observe that Corollary \ref{cor:subshifts} follows from the results above for the various sectors upon expressing the weights
$$\alpha_n = \sqrt{\frac{K p^n + N}{K p^n + N+j-1}}, \quad n = 0, 1,2, \ldots $$
as
$$
\alpha_n = \sqrt{\frac{p^n + N/K}{p^n + (N+j-1)/K}}, \quad n = 0, 1,2, \ldots \, .
$$

\medskip
\subsection{Subsector VIIIA}

To show that the shifts arising from points in Subsector VIIIA are completely hyperexpansive, we employ a standard device from the literature (see, e.g., \cite{At}) to show that the moment sequence $\gamma$ is completely alternating by examining instead $\Delta \gamma = - \nabla \gamma$. \   Recall that a shift with moment sequence $\gamma = (\gamma_n)$ is completely hyperexpansive if $\gamma$ is completely alternating. \  It is well known that $\gamma$ is completely alternating if and only if $\Delta \gamma$ is completely monotone (\cite[Ch. 4, Lemma 6.3]{BCR}).

\begin{thm}
Let $(N,D)$ be a point in Subsector VIIIA (that is, the  subsector of Sector {\rm VIII} bounded by the rays $D = N$ and $D = p N$). \  Then the associated weighted shift is completely hyperexpansive.
\end{thm}

\begin{proof} \ As noted above it suffices to show that $\Delta \gamma$ is a completely monotone sequence. \  A computation shows that if we produce a sequence $w = (w_n)_{n=0}^\infty$ of ``weights squared'' by setting
$$w_n = \frac{(\Delta \gamma)_{n+1}}{(\Delta \gamma)_{n}}, \quad n=0, 1,2, \ldots$$
then one has
$$w_n = \frac{1}{p} \left(\frac{p^n + N}{p^n + D/p}\right),   \quad n=0,1,2, \ldots.$$
(The ``weights squared'' is because since we have not normalized to force $(\Delta \gamma)_{0} = 1$, these are not weights squared derived from a moment sequence in the usual sense.)  Since under our assumptions on the pair $(N,D)$ we have $(N, D/p)$ in Sector I, the sequence $\left(\frac{p^n + N}{p^n + D/p}\right)_{n=0}^\infty$ is log completely alternating, and then clearly the sequence $(w_n)$ is log completely alternating. \  It is a standard computation to show then that reconstituting the moments $(\Delta \gamma)_n$ yields a log completely monotone sequence, and since log completely monotone implies completely monotone the sequence $(\Delta \gamma)$ is completely monotone as required.
\end{proof}

\medskip

Again substantial testing using \textit{Mathematica} \cite{Wol} suggests that there are no other loci of complete hyperexpansivity (except the trivial case that along the line $D = N$ we obtain the unilateral shift).

\begin{remark}
It is obvious that the weights sequence associated with some point $(N,D)$ is the reciprocal of that associated with $(D,N)$ (reflection across the diagonal $D = N$). \  If we define the subsector IA of Sector {\rm I} to be the reflection of $VIIIA$ across $D= N$, we have that the shifts arising from these points are not only $\mathcal{MID}$, and not even only the stronger property of having weights squared interpolated by a Bernstein function, but with the yet stronger property of being the reciprocals of the weights squared of a completely hyperexpansive shift.
\end{remark}

It is reasonable to ask about the other sectors V, VI, and VII, and the other portions of Sectors IV and VIII. \  While some things can be said, with the exception of some consideration to follow of $k$--hyponormality in Sector VI, these are at the moment somewhat \textit{terra incognita}. \  By the reflection principle just mentioned and straightforward computations, we may say, for example, that shifts arising from points in Sector VI have a weights squared sequence which is log completely monotone. \  We can find in the literature no study of such shifts or operator theoretic properties that follow from this assumption.

Since the Aluthge transform maps the class of $\mathcal{MID}$ shifts bijectively to itself, it is reasonable to consider both $AT$ and $AT^{-1}$ of a GWRS. \  While the computation of $AT(W_\alpha)$ yields a new $\mathcal{MID}$ shift, what results is not recognizable as a familiar one. \  In \cite{BCE3} there is a formula for $AT^{-1}$, but here again what results is, while new, not recognizable in some tidy way.

\medskip
\subsection{Completions}

It is a standard problem to consider some initial (finite) sequence of moments, and ask whether the sequence can be completed to the moment sequence of a weighted shift with some property;  the standard example is that the shift be subnormal. \  Of course this may instead be phrased as completing a weight sequence to that of a subnormal shift (see for example \cite{St}, \cite{CF1}, and subsequent papers). \   We turn next to showing that the class of GRWS is sufficiently rich to provide, in fact, numerous completions of the first three moments of a contractive shift arising from a two-atomic Berger measure with a mass at $1$. \  Recall that $\delta_x$ denotes the usual Dirac mass at $x$.

We first show that with a natural restriction on the support set there is only one completion.

\begin{prop}
Let $\gamma$ be the moment sequence arising from the Berger measure $\frac{1}{1 + a}\left(\delta_1 + a \delta_r\right)$, where $0 < r < 1$ and $a > 0$. \  Clearly we may write $r = \frac{1}{p}$ for some $p > 1$. \ Then the initial moment sequence $\gamma_0, \gamma_1, \gamma_2$ can be completed to the moment sequence of a GRWS with exponential parameter $p$ only by taking $N = \frac{a}{p}$ and $D = a$, and this completion is exactly the original shift and Berger measure.
\end{prop}

\begin{proof} \  One computes readily that $\gamma_0 = 1$, $\gamma_1 = \frac{1 + \frac{a}{p}}{1 + a}$, and $\gamma_2 = \frac{1 + \frac{a}{p^2}}{1 + a}$. \  It is then just another computation, using the explicit form of the weights, and hence the moments, for a GRWS with parameters $p$, $N$, and $D$ to obtain the result.
\end{proof}

If we permit ourselves to use a different geometric parameter $q$, there is a much richer collection of completions.

\begin{thm}
Let $\gamma$ be the moment sequence arising from the Berger measure $\frac{1}{1 + a}\left(\delta_1 + a \delta_\frac{1}{p}\right)$, where $p > 1$ and $a > 0$. \  Then there is an infinite family of completions of the initial moment sequence $\gamma_0, \gamma_1, \gamma_2$ to a GRWS with parameters $q \neq p$, $N$, and $D$, which may be parametrized by $N$, such that
\begin{itemize}
  \item If $-1 < N \leq \frac{a-a p}{(a+1) p}$ then the resulting GRWS is from Sector {\rm I} ($\mathcal{MID}$ with weights squared interpolated by a Bernstein function);
  \item If $\frac{a-a p}{(a+1) p} < N \leq \frac{a-a p}{a p+a+2 p}$ then the resulting GRWS is from Sector {\rm II} ($\mathcal{MID}$ with weights squared interpolated by a log Bernstein function);
  \item If $\frac{a-a p}{a p+a+2 p} < N \le 0$ then the resulting GRWS is from Sector {\rm III} and is subnormal.
\end{itemize}
\end{thm}

\begin{proof} \ One computes, as in the proof of the previous proposition and by comparing moments, that
$$q = \frac{a+\text{N} p^2-\text{N} p+p^2}{a+p} \quad \mbox{\rm and} \quad D = \frac{a \text{N} p+a p-a+\text{N} p}{a+p}.$$
There are various things to ensure, and it helps to observe that both $q$ and $D$ are increasing in $N$. \  Straightforward computations show that to obtain $q > 1$ requires $N > -1$; that $D > N$ if $N > -1$;  that $D \leq 0$ (so we are in Sector I) if $-1 < N \leq \frac{a-a p}{(a+1) p}$, where we note that $\frac{a-a p}{(a+1) p} = \frac{a}{1+a} \cdot \frac{1-p}{p}$ is strictly between $-1$ and $0$ using our assumptions on $a$ and $p$; that $0 < D \leq -N$ (so we are in Sector II) if $\frac{a-a p}{(a+1) p} < N \leq \frac{a-a p}{a p+a+2 p}$, noting that this is a nontrivial interval contained in $(-1, 0)$ from our assumptions on $a$ and $p$;  that $-N < D \leq 1$ (so we are in Sector III) if $\frac{a-a p}{a p+a+2 p} < N \le 0$.
\end{proof}

We leave to the interested reader the computation of a specific example (which may be chosen ``at random'' if the statements in Conjecture \ref{conjectures} hold, but in any event such an example exists) such that the same initial moment sequence $\gamma_0, \gamma_1, \gamma_2$ may be completed to various flavors of $\mathcal{MID}$ as well as subnormal but not $\mathcal{MID}$ and to both finitely and countably atomic support sets. \  Further, the class of GRWS is itself sufficiently rich to provide such completions.

We remark as well in passing that one cannot complete, with a GRWS, the initial three-term moment sequence arising from a two-atomic Berger measure with atoms at $0$ and $1$ if the density at $0$ is positive.

If one considers the initial segment $\gamma_0, \gamma_1, \gamma_2, \gamma_3$ arising from a three atomic contractive Berger measure with an atom at $1$, we have been unable to find any completion from among the GRWS unless the initial moment sequence is from a shift which is a GRWS arising from some $p > 1$ and a point $(N,D)$ in Sector IV along the line $D = p^2 N$;  in that case, there is no solution in $q$, $M$, and $P$ except the obvious solution $q = p$, $M = N$, and $P = D$.

\medskip
\subsection{Sector {\rm IV} and $k$--hyponormality}

We now turn to the task of proving that the special lines $D = p^k N$ subdivide Sector IV into subsectors in which we have $k$--hypormality but not $(k+1)$--hyponormality. \  Let $(N,D)$ be a point in Sector IV, so $0 < N < D < 1$. \  The discussion of $k$--hyponormality requires that we consider the positive (semi-)definiteness of families of matrices of the form
$$
\hat{M}_{\gamma}(k,j) = \left(
\begin{array}{cccc}
\gamma _{j} & \gamma _{j+1} & \cdots & \gamma _{j+k} \\
\gamma _{j+1} & \gamma _{j+2} & \cdots & \gamma _{j+k+1} \\
\vdots & \vdots & \ddots & \vdots \\
\gamma _{j+k} & \gamma _{k+j+1} & \cdots & \gamma _{j+2k}%
\end{array}
\right),
$$
where $k$--hyponormality is the positive definiteness, for each $j$, of this $(k+1) \times (k+1)$ matrix and $\gamma$ is the moment sequence of the relevant shift. \  In fact, it will be convenient to factor out $\gamma_j$ (which will surely not affect positivity) and we shall actually work with the $k$ by $k$ matrices (with $k \geq 2$)
$$
M(k,j) = \left(
\begin{array}{cccc}
1 & \alpha^2_{j} & \cdots & \alpha^2_{j}\cdots \alpha^2_{j+k-2}\\
\alpha^2_{j} & \alpha^2_{j}\alpha^2_{j+1} & \cdots & \alpha^2_{j}\cdots \alpha^2_{j+k-1} \\
\vdots & \vdots & \ddots & \vdots \\
\alpha^2_{j}\cdots \alpha^2_{j+k-2} & \alpha^2_{j}\cdots \alpha^2_{j+k-1} & \cdots & \alpha^2_{j}\cdots \alpha^2_{j+2k-3}%
\end{array}
\right).
$$
(To move to matrices of size $k$ will ease the notation, but at the cost of remembering that positivity of the matrix yields $(k-1)$--hyponormality.) We will, as usual, determine the positivity of the matrices by using the Nested Determinant Test, which will require us to obtain a (reasonably) tractable form for the determinant.

We will obtain the determinants using what is called ``The Condensation Method'' in \cite{Kr} (which discusses possible original sources for the method), and we begin with its Proposition 10.

\begin{prop}\cite[Prop. 10]{Kr}
Let $A$ be an $n$ by $n$ matrix. \  Denote the submatrix of $A$ in which rows $i_1, i_2, \ldots, i_k$ and columns $j_1, j_2, \ldots, j_k$ are omitted by $A^{j_1, j_2, \ldots, j_k}_{i_1, i_2, \ldots, i_k}$. \  Then there holds
\begin{equation}\label{eq:basicKratt}
\det A \cdot \det A^{1,n}_{1,n} = \det A^1_1
\cdot \det A^n_n - \det A^n_1 \cdot \det A^1_n .
\end{equation}
\end{prop}
For our matrices $M(k,j)$ we have the following useful identities:
\begin{enumerate}
  \item $M(k,j)^1_1 = \alpha^2_j \alpha^2_{j+1} M(k-1, j+2)$;
  \item $M(k,j)^k_k =  M(k-1, j)$;
  \item $M(k,j)^1_k = M(k,j)^k_1 = \alpha^2_j M(k-1, j+1)$;
  \item $M(k,j)^{1,k}_{1,k}= \alpha^2_j \alpha^2_{j+1} M(k-2, j+2)$.
\end{enumerate}
Upon insertion of the definitions of the $\alpha_j$, translating to determinants, and using the identities just above, we have that, for each $k$ and $j$,
\begin{eqnarray}
\det M(k,j) \cdot  \left(\frac{p^j+N}{p^j+D}\right)^{k-2}&\cdot & \left(\frac{p^{j+1}+N}{p^{j+1}+D}\right)^{k-2} \cdot \det M(k-2, j+2)     \label{eq:Krattforus} \nonumber \\
&& \hspace*{-1in}= \left(\frac{p^j+N}{p^j+D}\right)^{k-1} \left(\frac{p^{j+1}+N}{p^{j+1}+D}\right)^{k-1} \cdot \det M(k-1, j+2) \cdot \det M(k-1, j)  \\
&&  \hspace*{-.8in} - \left(\frac{p^j+N}{p^j+D}\right)^{2 (k-1)} \cdot (\det M(k-1, j+1))^2.  \nonumber
\end{eqnarray}

We sketch below the rather laborious (weak) induction to give the proof of a closed form for $\det M(k, j)$, for such points $(N,D)$, conjectured with the aid of \textit{Mathematica} \cite{Wol} and the computing cluster at Bucknell University.

\begin{lem}  \label{le:detMNj}
For $k \geq 2$ and $j \geq 0$ let $M(k, j)$ be the matrix
$$
M(k,j) = \left(
\begin{array}{cccc}
1 & \alpha^2_{j} & \cdots & \alpha^2_{j}\cdots \alpha^2_{j+k-2}\\
\alpha^2_{j} & \alpha^2_{j}\alpha^2_{j+1} & \cdots & \alpha^2_{j}\cdots \alpha^2_{j+k-1} \\
\vdots & \vdots & \ddots & \vdots \\
\alpha^2_{j}\cdots \alpha^2_{j+k-2} & \alpha^2_{j}\cdots \alpha^2_{j+k-1} & \cdots & \alpha^2_{j}\cdots \alpha^2_{j+2k-3}%
\end{array}
\right),
$$
where the $\alpha_j$ are the weights arising from the point $(N,D)$ with $0 < N < D < 1$. \  We have, for $j = 0, 1, 2, \ldots$,
\begin{equation}  \label{eq:detM2j}
\det M(2,j) = \frac{ p^j (N-D) (1-p) \left(N + p^j\right)}{\left(D + p^j\right)^2 \left(D + p^{j+1}\right)}
\end{equation}
and
\begin{equation}  \label{eq:detM3j}
\det M(3,j) = \frac{ p^{3 j+2} (1-p)^2 (1-p^2)  (N-D)^2 (N p - D) \left(N + p^j\right)^2 \left(N + p^{j+1}\right)}{\left(D + p^j\right)^3 \left(D + p^{j+1}\right)^3 \left(D + p^{j+2}\right)^2 \left(D + p^{j+3}\right)}.
\end{equation}
For $k \geq 4$ and $j \geq 0$, $\det M(k,j)$ has the form
\begin{eqnarray}  \label{eq:detMkj}
\quad & p^{\frac{1}{3}k(k-1)(k-2)} \left(p^j\right)^{\frac{1}{2}k(k-1)} \cdot \dfrac{ \left(\prod _{i=0}^{k-2} \left(1-p^{i+1}\right)^{k-i-1}\right) \left(\prod _{i=0}^{k-2} \left(N p^i - D\right)^{k-i-1}\right)}
{\left(\prod _{\ell=0}^{k-2} \left(D + p^{\ell+j}\right)\right)^k } \times \nonumber \\
& \times \dfrac{\left(\prod _{i=0}^{k-2} \left(N + p^{i+j}\right)^{k-i-1}\right)}{\prod _{i=1}^{k-1} \left(D + p^{j+i+k-2}\right)^{k-i}}.
\end{eqnarray}

\mk
(Note that the expressions for $\det M(2,j)$ and $\det M(3,j)$ fit the general form, and are merely illustrative.)
\end{lem}

\noindent Proof. \  Direct computations of the relevant determinants for $k=2$ and $k=3$ are routine, and form the basis for the (weak) induction on $k$. \  (In giving the results in \eqref{eq:num1RHS}, \eqref{eq:denom1RHS}, \eqref{eq:num2RHS}, and \eqref{eq:denom2RHS} we have boxed certain terms (or the absence of a term) to facilitate the subsequent discussion.)  It is helpful to prove, as part of that induction, that the first term on the right-hand side of \eqref{eq:Krattforus} has numerator of the form
\begin{eqnarray}
p^{\frac{1}{3} (k-2) (k-1) (2 k-3)} \cdot  p^{j (k-1) (k-2)} &\cdot & \prod _{i=0}^{k-3} \left(1-p^{i+1}\right)^{2 (k-2-i)} \, \prod _{i=0}^{k-3} \left(N p^i - D\right)^{2 (k-2-i)} \cdot \nonumber \label{eq:num1RHS} \\
&\cdot& \boxed{\left(N + p^j\right)^{2 k-3}} \, \prod _{i=0}^{k-3} \left(N + p^{i+j+1}\right)^{2 (k-2-i)} \boxed{(N+p^{j+k-1})}
\end{eqnarray}
and denominator of the form
\begin{equation}  \label{eq:denom1RHS}
 \left(\prod _{i=0}^{k-3} \left(D+p^{i+j}\right)\right)^{2 (k-1)}  \boxed{\left(D+p^{j+k-2}\right)^{2 k-3}} \left(\prod _{i=0}^{k-3} \left(D+p^{i+j+k-1}\right)^{2 k-4-2 i}\right)  \boxed{\left(D+p^{j+2 k-3}\right)}.
 \end{equation}
(In obtaining the denominator it is useful to note that, in the denominator
\begin{eqnarray*}
(D + p^j)^{k-1} (1 + p^{j+1})^{k-1} &\cdot & \left(\prod_{i=0}^{k-3}(D + p^{j+2 + i})\right)^{k-1} \cdot \\
&\cdot & \prod_{i=1}^{k-2} (D + p^{j+2+i+k-1-2})^{k-1-i} \cdot \\
&\cdot& \left(\prod_{i=0}^{k-3}(D + p^{j+i})\right)^{k-1} \prod_{i=1}^{k-2} (D + p^{j+i+k-1-2})^{k-1-i}\\
\end{eqnarray*}
of this first term on the right-hand side arising from the product, the first of the terms contributes only to $(D + p^j)$, the second only to $(D + p^{j+1})$, the fifth only to terms $(D + p^j)$ through $(D + p^{j+k-3})$, the third only to terms $(D + p^{j+2})$ through $(D + p^{j+k-1})$, the sixth only to terms $(D + p^{j+k-3})$ through $(D + p^{j+2k-5})$, and the fourth only to terms $(D + p^{j+k-1})$ through $(D + p^{j+2k-3})$.)

Similarly, the second term on the right-hand side of \eqref{eq:Krattforus} has numerator of the form
\begin{eqnarray}
p^{\frac{1}{3} (k-2) (k-1) (2 k-3)} \cdot  p^{j (k-1) (k-2)} &\cdot & \prod _{i=0}^{k-3} \left(1-p^{i+1}\right)^{2 (k-2-i)} \, \prod _{i=0}^{k-3} \left(N p^i - D\right)^{2 (k-2-i)} \cdot  \nonumber\\
&\cdot& \boxed{\left(N + p^j\right)^{2 k-2}} \, \prod _{i=0}^{k-3} \left(N + p^{i+j+1}\right)^{2 (k-2-i)} \framebox[15mm]{\rule{0pt}{5mm}}  \label{eq:num2RHS}
\end{eqnarray}
and denominator of the form
\begin{equation}  \label{eq:denom2RHS}
\left(\prod _{i=0}^{k-3} \left(D+p^{i+j}\right)\right)^{2 (k-1)} \boxed{\left(D+p^{j+k-2}\right)^{2 k-2}} \left(\prod _{i=0}^{k-3} \left(D+ p^{i+j+k-1}\right)^{2 k-4- 2 i}\right) \framebox[15mm]{\rule{0pt}{5mm}}.
 \end{equation}

The expressions for the numerators in \eqref{eq:num1RHS} and \eqref{eq:num2RHS}, and those for the denominators in \eqref{eq:denom1RHS} and \eqref{eq:denom2RHS}, differ only in a few terms, and (as noted above) to facilitate the comparison we have indicated by boxes those places where the terms are not equal. \  If we set $K$ to be what is in common in the two terms on the right-hand side of \eqref{eq:Krattforus}, we see that $K$ is given as
$$
\scriptstyle{\frac{p^{\frac{1}{3} (k-2) (k-1) (2 k-3)} \cdot  p^{j (k-1) (k-2)} \cdot  \prod _{i=0}^{k-3} \left(1-p^{i+1}\right)^{2 (k-2-i)} \, \prod _{i=0}^{k-3} \left(N p^i - D\right)^{2 (k-2-i)}
\cdot \left(N + p^j\right)^{2 k-3} \, \prod _{i=0}^{k-3} \left(N + p^{i+j+1}\right)^{2 (k-2-i)}}{\left(\prod _{i=0}^{k-3} \left(D+p^{i+j}\right)\right)^{2 (k-1)}  \left(D+p^{j+k-2}\right)^{2 k-3} \left(\prod _{i=0}^{k-3} \left(D+p^{i+j+k-1}\right)^{2 k-4-2 i}\right)}}.
$$
We readily see that the right-hand side of \eqref{eq:Krattforus} has the form
\begin{eqnarray}
K \cdot\left(\frac{(N+p^{j+k-1})}{\left(D+p^{j+2 k-3}\right)}-\frac{\left(N + p^j\right)}{\left(D+p^{j+k-2}\right)}\right)&=& K \cdot \left( -\frac{p^{j+2} \left(p-p^k\right) \left(p^2 D-N p^k\right)}{\left(p^{j+k}+p^2 D\right) \left(p^{j+2 k}+p^3 D\right)}\right) \nonumber \\
&=& K \cdot \frac{p^j (1 - p^{k-1})(N p^{k-2} - D)}{(D + p^{j+k-2})(D + p^{j + 2k - 3})}.  \label{eq:RHSofKrattforus}
\end{eqnarray}
The verification that $\det M(k,j)$ is as in \eqref{eq:detMkj} is now a computation from \eqref{eq:Krattforus}, \eqref{eq:RHSofKrattforus}, and the induction hypothesis for the form of $\det M(k-2, j+2)$.  \pfend

With the determinant formula in hand, we may determine the divisions of Sector IV into subsectors on which we have various $k$--hyponormalities. \  Observe that if we are at any point where, for some $\ell = 0, 1, 2, \ldots$, $D = p^\ell N$, we will have the term $N p^\ell - D = 0$. \  In this case it is easy to show that the relevant shift is finitely atomic and subnormal, either by the trivial modification to the proof in \ref{th:fintomicSIV} or by realizing the moment matrices are all of finite rank and ``flat'' (see \cite{CF1}).

\begin{thm}
Let $p>1$ and suppose  $0 < N < D < 1$ and $D \neq p^k N$, $k = 0, 1,2, \ldots$. \  For the GWRS $W_\alpha$ arising from $p$ and the point $(N, D)$, and for any $k = 1, 2, \ldots$, $W_\alpha$ is $k$--hyponormal if and only if $D > p^{k-1} N$. \  It follows that if $N p^{k-1}< D < N p^k$, then $W_\alpha$ is $k$--hyponormal but not $(k+1)$--hyponormal.
\end{thm}

\begin{proof} \  Recall that for $k$--hyponormality we require that the matrices $M(n, m)$ be positive (semi-) definite for $n = 2, \ldots, k + 1$ and for $m = 0, 1,2, \ldots$. \  One sees readily that with the assumption $D \neq p^k N$, $k = 0, 1, 2, \ldots$, none of the relevant determinants are zero and thus we may detect matrix positivity by determinant positivities, using the Nested Determinant Test.

The claim is obvious for $k = 1$ from  \eqref{eq:detM2j}, for $k = 2$ from  \eqref{eq:detM3j}, and for $k = 3$ from \eqref{eq:detMkj} in the case $k = 4$;  we proceed using induction on $k$ starting from this. \  Suppose then that the claim holds for some $k > 4$. \  For $(k + 1)$--hyponormality we must have $k$--hyponormality and also examine the determinant of $M(k + 2, m)$ for all $m$. \  In the expression \eqref{eq:detMkj} for $\det M(k,j)$ it is clear that only the terms
$$
\left(\prod _{i=0}^{k-2} \left(1-p^{i+1}\right)^{k-i-1}\right) \left(\prod _{i=0}^{k-2} \left(N p^i-D\right)^{k-i-1}\right)
$$
can possibly induce a negative term. \  There are clearly $2(k-1)$ such potentially negative terms, raised of course to various powers. \ Under the assumption of $k$--hyponormality, we have the total sign of the $2*(k+1-1)= 2k$ such terms (with their powers) in $\det M(k+1, m)$ as positive for all $m$. \  If we move to the determinant of $M(k + 2, m)$, we have each of these previous terms with its power increased by one, but this induces no sign change from the determinant of $M(k + 1, m)$;  we also have two ``new'' terms appearing, namely $1-p^{k+2-1} = 1 - p^{k+1}$ and $ N p^{k+2-2} - D = N p^{k+1-1} - D$. \  The first of these is assuredly negative;  in order that the total sign be positive, the second must be as well, which is exactly $D > N p^{k+1-1}$ as as claimed to be required for $(k+1)$--hyponormality. \ It is clear that this requirement is both necessary and sufficient for $(k+1)$--hyponormality, since we already have $D > N p^{k-1}$.

The second claim follows immediately from the first.
\end{proof}

We remark that while the determinant formula in Lemma \ref{le:detMNj} clearly involves $j$, the sign of the determinant for a particular $k$, is uniform in $j$.

Observe also that many of the results carry implicitly that properties are \textbf{radial} within the unit square, in the sense that if the shift associated with $(N,D)$ has some property then so does the shift associated with $(\lambda N, \lambda D)$ for $\lambda \geq 0$ and such that $(\lambda N, \lambda D)$ remains in the unit square. \  It is to be pointed out that many of the proofs, in fact, establish this explicitly, and this is somewhat remarkable in that scaling of $(N,D)$ does not amount merely to scaling the weights (or moments). \  In the determinant formula in Lemma \ref{le:detMNj}, for example, examination of the result of replacing $(N,D)$ by some $(\lambda N, \lambda D)$ remaining in the unit square shows that this does not change the sign of the determinant:  terms of the form
$$\left(\lambda N p^j - \lambda D\right)^{k-j-1}$$
are homogeneous in $\lambda$ and do not change sign, and the signs of other terms that could conceivably change do not with the standing assumption that $p > 1$. \ Similarly, Theorem \ref{th:1ovpxmaisCM}, Lemma \ref{le:afnisCM}, and Theorem \ref{th:pxpaovpxpblogB} are unchanged by a radial change in relevant constants, and so the properties in Sector I of weights squared interpolated by a Bernstein function and in Sector II of weights squared interpolated by a log Bernstein function are intrinsically radial. \  Of course subnormality, as $k$--hyponormality for all $k$, is likewise radial.

This point of view provides additional evidence that the shifts in Sector III, known to be subnormal, are not $\mathcal{M I D}$. \  In experiments with \textit{Mathematica} \cite{Wol} for $p = 3/2$, for example, we can discard much of Sector III from being $\mathcal{M I D}$ because we can check on the upper boundary $D = 1$ that the relevant shifts do not even have a subnormal (Schur) square root.

\medskip

\subsection{Affine subshifts of subnormal and $\mathcal{MID}$ shifts}  \label{subse:remarks}

Certain of the GRWS provide tractable examples of shifts whose affine subshifts need not be subnormal.  Let (affine)  $g$ be of the form $g(n) = \ell n + r$, with $\ell$ a positive integer and $r$ a non-negative integer. The associated subshift of the GRWS with weights
$$\alpha_n = \sqrt{\frac{p^n + N}{p^n + D}}$$
is then the shift with weights
\begin{eqnarray*}
\hat{\alpha}_n &=& \sqrt{\frac{p^{\ell n+ r} + N}{p^{\ell n+ r} + D}}\\
&=&\sqrt{\frac{p^{\ell n} + \frac{N}{p^r}}{p^{\ell n}+ \frac{D}{p^r}}}\\
&=&\sqrt{\frac{{(p^{\ell})}^n + \frac{N}{p^r}}{{(p^{\ell})}^n+ \frac{D}{p^r}}}.
\end{eqnarray*}
This is another GRWS, with parameter $p^\ell$ and corresponding to the point $(\frac{N}{p^r}, \frac{D}{p^r})$ in the associated Magic Square, and so its subnormality or lack thereof is understood.  In particular, it is easy to compute that starting on one of the special lines $D = p^k N$ in Sector IV for the original there are choices of $k$ and $\ell$ so that the resulting shift is in Sector IV, but not on one of the special lines, and is therefore not subnormal.  In fact, one may choose $k$ and $\ell$ so as to have the resulting shift $j$-hyponormal but not $(j+1)$-hyponormal for any $j \geq 1$.  Thus there is a rich collection of examples showing to what degree a subnormal shift may lack subnormality of an affine subshift.

Observe that the same computation shows that if we begin with a GRWS in Sectors I or II, and hence $\mathcal{MID}$, the resulting affine subshift is again in the same sector (but now corresponding of course to $p^\ell$) and is therefore still $\mathcal{MID}$.  Starting in Sector III, the new shift is again in Sector III and so subnormality was retained even in the absence of the stronger $\mathcal{MID}$ property.  The preservation of subnormality under the taking of an affine subshift is therefore, in general, quite wide open.

However, for the class of $\mathcal{MID}$ shifts, the conjecture suggested by the behavior of the $\mathcal{MID}$ GRWS is correct:  $\mathcal{MID}$, and thereby subnormality, is preserved under the taking of affine subshifts.

We first require a lemma.  Let $T$ be a bounded linear operator on $\mathcal{H}$.
Consider $M_{T}$,  an operator on $\mathcal{L}(\mathcal{H})$, defined by $M_{T}(X)=T^*XT$.  We have the following.

\begin{lem}
With $M_{T}$ as above, for every $X\in \mathcal{L}(\mathcal{H})$ and every integer i, $M_{T}^i(X)=(T^i)^*XT^i$ and thus $M_{T}^i=M_{T^i}$.  Further,
\begin{equation}  \label{eq:variousM}
(I-M_{T^k})^n = \sum_{j=0}^{n(k-1)} c_j  \big(  \sum_{l=0}^{n} (-1)^l\binom{n}{l}M_{T^{l+j}} \big)
\end{equation}
where we do not specify the $c_j$ except to note that they are positive.  Finally, for fixed integers $k$ and $i_0$,
\begin{equation}\label{eq PG}
\sum_{i=0}^{n}(-1)^i\binom{n}{i} (T^*)^{k(i+m)+i_0}T^{k(i+m)+i_0}
=
\sum_{j=0}^{n(k-1)} c_j  \Big(  \sum_{l=0}^{n} (-1)^l\binom{n}{l}   (T^*)^{km+l+j+i_0}T^{km+l+j+i_0}        \Big).
\end{equation}
\end{lem}

\noindent Proof.  The first claim is just a computation.   For the second, with a fixed integer $k \geq 2$, we may compute
$$
\begin{array}{lcl}
(I-M_{T^k})^n &= & \sum_{i=0}^{n}(-1)^i\binom{n}{i} M_{T^{ki}} \\
   &= & (I-M_{T}^k)^n \\
   &= & \big(\sum_{j=0}^{k-1} M_{T^j} \big)^n (I-M_{T})^n  \\
 &= & \sum_{j=0}^{n(k-1)} c_j M_{T^j}   \sum_{l=0}^{n} (-1)^l\binom{n}{l}M_{T^l}  \\
  &= & \sum_{j=0}^{n(k-1)} c_j  \big(  \sum_{l=0}^{n} (-1)^l\binom{n}{l}M_{T^{l+j}} \big) \\
\end{array}
$$
and observe the $c_j$ are positive as asserted.

Using the second claim as applied to the identity operator $I$ yields
$$\sum_{i=0}^{n}(-1)^i\binom{n}{i} (T^*)^{ki}T^{ki}
=
\sum_{j=0}^{n(k-1)} c_j  \Big(  \sum_{l=0}^{n} (-1)^l\binom{n}{l}   (T^*)^{l+j}T^{l+j}        \Big),$$
and multiplying this on the left by $(T^*)^{km+i_0}$ and on the right by $T^{km+i_0}$ yields the third claim.  \pfend

We next show that affine subshifts of $\mathcal{MID}$ shifts are $\mathcal{MID}$.  The motivation for the proof is to realize, using the characterization of $\mathcal{MID}$ shifts as those whose weights are log completely alternating, that the following sample computation works in general.  A ``test'' of some log alternating property for the subshift may be written in terms of ``tests'' for log completely alternating of the original shift:
$$\begin{array}{ccccccc}
\ln \alpha_1 & \hspace*{.2in}     &- 3\ln \alpha_3 &\hspace*{.2in}  & +3 \ln \alpha_5 & \hspace*{.2in} &-\ln \alpha_7 \hspace{.1in} =\\
&&&&&&\\
(\ln \alpha_1& -3 \ln \alpha_2& +3 \ln \alpha_3&-\ln \alpha_4)&&& \\
&&&&&&\\
+ &3(\ln \alpha_2&-3 \ln \alpha_3& + 3 \ln \alpha_4 & - \ln \alpha_5) &&\\
&&&&&&\\
&+ &3(\ln \alpha_3 & -3 \ln \alpha_4 & +3 \ln \alpha_5 & - \ln \alpha_6)& \\
&&&&&&\\
&&+&(\ln \alpha_4 &-3 \ln \alpha_5 & +3 \ln \alpha_6& - \ln \alpha_7).\\
\end{array}$$

\begin{thm}
Let $W_\alpha$ be an $\mathcal{MID}$ weighted shift.  Then any affine subshift of $W_\alpha$ is $\mathcal{MID}$.
\end{thm}

\noindent Proof.  We use the characterization of $\mathcal{MID}$ as those whose weights are log completely alternating, and the task is to justify in general the claim for the motivating computation above.  Consider $W_\alpha$ scaled so that all the weights are strictly less than $1$;  this will not change whether the shift or an affine subshift is $\mathcal{M I D}$, and we will retain $\alpha$ for the sequence of weights so scaled. Recall that under our conventions each $\alpha_n$ is strictly positive;  then each $\ln \alpha_n$ is negative.  Consider the sequence $\beta = \{\beta_n\}_{n=0}^\infty$ defined by $\beta_n := \frac{\ln \alpha_n}{\ln \alpha_0}$.  Fix $i_0$ a non-negative integer and $k\geq 2$ a positive integer and let the subshift of $W_\alpha$ be induced by $g(j) = k j + i_0$.  Then the induced subsequence of the weights is $\tilde{\alpha}$ defined by $\tilde{\alpha}_n = \alpha_{k n + i_0}$.  Clearly $\widetilde{\beta}$ defined by $\widetilde{\beta}_n =\frac{\ln \alpha_{k n + i_0}}{\ln \alpha_0}$ is the affine subsequence of $\beta$ induced by the same $g$.

With $T$ the weighted shift having moment sequence $\beta$, and applying \eqref{eq PG} as a quadratic form to the vector $e_0$ in the canonical basis for the acting space of $T$, after converting to the moments $\beta$ we obtain
$$\nabla^n \widetilde{\beta} (m)=\sum_{i=0}^n (-1)^i \binom{n}{i}\beta_{k(m+i)+i_0}=
\sum_{j=0}^{n(k-1)} c_j  \nabla^m \beta(km+j+i_0).
$$
Multiplying both sides by $\ln \alpha_0$ yields that
$$\nabla^n \ln \widetilde{\alpha} (m)=\sum_{i=0}^n (-1)^i \binom{n}{i} \ln \alpha_{k(m+i)+i_0}=
\sum_{j=0}^{n(k-1)} c_j  (\nabla^m \ln \alpha)(km+j+i_0).
$$
Recalling the $c_j$ are positive, we see that since each $(\nabla^m \ln \alpha)(km+j+i_0)$ is non-positive (since $(\alpha)$ is log completely alternating), $\nabla^n \ln \widetilde{\alpha} (m)$ is non-positive.  Since this holds for each positive integer $n$ and non-negative integer $m$, the sequence $\tilde{\alpha}$ is log completely alternating as desired.  \pfend

\medskip

We note again, using the example of a GRWS from Sector III, that subnormality of affine subshifts may be preserved in some cases even without the stronger $\mathcal{MID}$ property.

\vspace{.1in}

\noindent {\bf Acknowledgments}. \ The authors wish to express their appreciation for support and warm hospitality during various visits (which materially aided this work) to Bucknell University, the University of Iowa, and the Universit\'{e} des Sciences et Technologies de Lille, and particularly the Mathematics Departments of these institutions. \ Several examples in this paper were obtained using calculations with the software tool \textit{Mathematica} \cite{Wol}. \ We wish as well to thank Michael Harvey, of Bucknell University, for his generous assistance in using the computer cluster to gather examples leading to the correct conjecture for the general determinant formula in Lemma \ref{le:detMNj}.

\end{document}